\documentclass[11pt,a4paper]{article}

\usepackage{epsf,epsfig,amsfonts,amsgen,amsmath,amstext,amsbsy,amsopn,amsthm,cases,listings,color
}
\usepackage{ebezier,eepic}
\usepackage{color}
\usepackage{multirow}
\usepackage{epstopdf}
\usepackage{graphicx}
\usepackage{pgf,tikz}
\usepackage{mathrsfs}
\usepackage[marginal]{footmisc}
\usepackage{enumitem}
\usepackage[titletoc]{appendix}
\usepackage{booktabs}
\usepackage{url}
\usepackage{subfigure}
\usepackage{mathrsfs}
\usetikzlibrary{arrows}

\allowdisplaybreaks[1]
\usepackage{aligned-overset}

\definecolor{uuuuuu}{rgb}{0.27,0.27,0.27}
\definecolor{sqsqsq}{rgb}{0.1255,0.1255,0.1255}

\newtheorem{definition}{Definition} [section]
\newtheorem{theorem}[definition]{Theorem}
\newtheorem{lemma}[definition]{Lemma}

\newtheorem{claim}[definition]{Claim}

\newenvironment{pf}{\noindent{\bf Proof}}

\begin{document}
\title{\bf\Large Stability theorems for some Kruskal-Katona type results}

\date{\today}

\author{
Xizhi Liu
\thanks{Department of Mathematics, Statistics, and Computer Science, University of Illinois, Chicago, IL, 60607 USA. Email: xliu246@uic.edu.
Research partially supported by NSF awards DMS-1763317 and DMS-1952767.}
\and
Sayan Mukherjee
\thanks{Department of Mathematics, Statistics, and Computer Science, University of Illinois, Chicago, IL, 60607 USA. Email: smukhe2@uic.edu.}
}
\maketitle
\begin{abstract}
The classical Kruskal-Katona theorem gives a tight upper bound for the size of
an $r$-uniform hypergraph $\mathcal{H}$ as a function of the size of its shadow.
Its stability version was obtained by Keevash who proved that if the size of $\mathcal{H}$
is close to the maximum, then $\mathcal{H}$ is structurally close to a complete $r$-uniform hypergraph.
We prove similar stability results for two classes of hypergraphs
whose extremal properties have been investigated by many researchers:
the cancellative hypergraphs and hypergraphs without expansion of cliques.

\end{abstract}

\section{Introduction}\label{SEC:introduction}
An $r$-uniform hypergraph (henceforth $r$-graph) $\mathcal{H}$ is a collection of $r$-subsets of a ground set $V(\mathcal{H})$,
which is called the vertex set of $\mathcal{H}$. The $r$-sets that are contained in $\mathcal{H}$ are called edges of $\mathcal{H}$,
and we identify $\mathcal{H}$ with its edge set.
We use $K_{n}^{r}$ to denote the complete $r$-graph on $n$ vertices.
The {\em shadow} $\partial\mathcal{H}$ of an $r$-graph $\mathcal{H}$ is an $(r-1)$-graph defined as
\begin{align}
\partial\mathcal{H} = \left\{A\in \binom{V(\mathcal{H})}{r-1}\colon \exists B\in \mathcal{H}\text{ such that }A\subset B\right\}. \notag
\end{align}
The classical Kruskal-Katona theorem \cite{KR63,KA66} gives a tight upper bound for $|\mathcal{H}|$
as a function of $|\partial\mathcal{H}|$ . Here we state a technically simper version due to Lov\'{a}sz.

\begin{theorem}[Lov\'{a}sz {\cite[Ex. 13.31(b)]{LO93}}]\label{THM:KK-Lovasz-version}
Let $r\ge 2$ be an integer and $\mathcal{H}$ be an $r$-graph with $|\partial\mathcal{H}| = \binom{x}{r-1}$ for some real number $x\ge r$.
Then $|\mathcal{H}| \le \binom{x}{r}$.
Moreover, equality holds if $x$ is an integer and $\mathcal{H}$ is a union of $K_{x}^{r}$ and a set of isolated vertices.
\end{theorem}

Keevash \cite{KE08} gave a nice short proof to Theorem~\ref{THM:KK-Lovasz-version} without using the compression technique,
and moreover, his proof was extended to obtain the following stability result.

\begin{theorem}[Keevash \cite{KE08}]\label{THM:Keevash-stability}
Let $r\ge 2$ be an integer.
For every $\delta > 0$ there exists $\epsilon > 0$ such that if $\mathcal{H}$ is an $r$-graph with
$|\partial\mathcal{H}| = \binom{x}{r-1}$ and $|\mathcal{H}| \ge (1-\epsilon)\binom{x}{r}$ for some real number $x\ge r$,
then there exists a set $V'\subset V(\mathcal{H})$ of size at most $\lceil x \rceil$ such that all but at most $\delta \binom{x}{r}$
edges of $\mathcal{H}$ are contained in $V'$.
\end{theorem}

In this work, we consider extensions of Theorem~\ref{THM:Keevash-stability} to the class of
hypergraphs that omit certain forbidden substructures.

Let $r\ge 2$ and $\mathcal{F}$ be a family of $r$-graphs.
An $r$-graph is {\em $\mathcal{F}$-free} if it does not contain any member of $\mathcal{F}$ as a (not necessarily induced) subgraph.
The {\em Tur\'{a}n number} ${\rm ex}(n,\mathcal{F})$ of $\mathcal{F}$ is the maximum size of an $\mathcal{F}$-free $r$-graph on $n$ vertices,
and the {\em Tur\'{a}n density} of $\mathcal{F}$ is $\pi(\mathcal{F}) = \lim_{n\to \infty}{\rm ex}(n,\mathcal{F})/\binom{n}{r}$.
It is one of the central problems in extremal combinatorics to determine ${\rm ex}(n,\mathcal{F})$ for various families $\mathcal{F}$.

Much is known about ${\rm ex}(n,\mathcal{F})$ when $r=2$ and one the most famous results in this regard is Tur\'{a}n's theorem \cite{TU41},
which states that for $\ell \ge 2$ the Tur\'{a}n number ${\rm ex}(n,K_{\ell+1})$ is uniquely achieved by
$T(n,\ell)$ which is the $\ell$-partite graph on $n$ vertices with the maximum number of edges.
However, for $r\ge 3$ determining ${\rm ex}(n,\mathcal{F})$, even $\pi(\mathcal{F})$, is notoriously hard in general.
Compared to the case $r=2$, very little is known about ${\rm ex}(n,\mathcal{F})$ for $r\ge 3$,
and we refer the reader to \cite{KE11} for results before 2011.

To have a better understanding of the extremal properties of $\mathcal{F}$-free hypergraphs,
in \cite{LM19A}, the following question which combines the Kruskal-Katona theorem and the hypergraph Tur\'{a}n problem
was studied systemically.
\begin{align}
&\mbox{If $\mathcal{H}$ is $\mathcal{F}$-free, what are the possible values of $|\mathcal{H}|$ for fixed $|\partial\mathcal{H}|$?} \notag
\end{align}

Tight upper bound for $|\mathcal{H}|$ was obtained in \cite{LM19A} for two specific families
that have been investigated by many researchers: cancellative hypergraphs and hypergraphs without expansions of cliques.

\subsection{Cancellative hypergraphs}\label{SUBSEC:cancellative-hypergraph}
For every integer $r\ge 2$ let $\mathcal{T}_{r}$ be the family of $r$-graphs with at most $2r-1$ vertices
and three edges $A,B,C$ such that $A\triangle B \subset C$.
Note that when $r=2$ the family $\mathcal{T}_{2}$ consists of only one graph $K_3$.
An $r$-graph $\mathcal{H}$ is {\em cancellative} if it has the property that $A\cup B = A\cup C$ implies $B = C$,
where $A,B,C$ are edges in $\mathcal{H}$.
It is easy to see that an $r$-graph is cancellative if and only if it is $\mathcal{T}_{r}$-free,
and a cancellative graph is simply a triangle-free graph.

Let $\ell\ge r \ge 2$ be integers and let $V_{1} \cup \cdots \cup V_{\ell}$ be a partition of $[n]$
with each $V_{i}$ of size either $\lfloor n/\ell \rfloor$ or $\lceil n/\ell \rceil$.
The {\em generalized Tur\'{a}n graph} $T_{r}(n,\ell)$ is the collection of all $r$-subsets of $[n]$ that have
at most one vertex in each $V_{i}$. Let $t_{r}(n,\ell) = |T_{r}(n,\ell)| \sim \binom{\ell}{r}(n/\ell)^r$.

In the 1960's, Katona raised the problem of determining the maximum size of a cancellative $3$-graph on $n$ vertices and
conjectured that the maximum size is achieved by $T_{3}(n,3)$.
Bollob\'{a}s \cite{BO74} proved Katona's conjecture
and he conjectured that a similar result holds for all $r\ge 4$.
Sidorenko \cite{SI87} proved it for $r=4$,
but Shearer \cite{SH96} gave a construction showing
that Bollob\'{a}s' conjecture is false for all $r\ge 10$.
The number $\pi(\mathcal{T}_{r})$ is still unknown for all $r\ge 5$.

The following Kruskal-Katona type result for cancellative $r$-graphs was proved in \cite{LM19A}
despite $\pi(\mathcal{T}_{r})$ is only known for $r\in\{3,4\}$.

\begin{theorem}[\cite{LM19A}]\label{THM:LM-cancellative}
Let $r \ge 2$ be an integer and $\mathcal{H}$ be a cancellative $r$-graph.
Suppose that $|\partial\mathcal{H}| = x^{r-1}/r^{r-2}$ for some real number $x\ge r$.
Then $|\mathcal{H}| \le (x/r)^{r}$. In other words, $|\mathcal{H}| \le \left(|\partial\mathcal{H}|/r\right)^{r/(r-1)}$.
\end{theorem}

For integers $n\ge m \ge \ell \ge r \ge 2$
let $T_{r}(n,m,\ell)$ be the union of $T_{r}(m,\ell)$ and a set of $n-m$ isolated vertices.
Notice that the inequality in Theorem~\ref{THM:LM-cancellative} is tight for $T_{r}(n,m,r)$ if $m$ is a multiple of $r$.

We prove a corresponding stability result for Theorem~\ref{THM:LM-cancellative}.

\begin{theorem}\label{THM:stability-cancellative-KK-type}
Let $r \ge 2$ be an integer.
For every $\delta > 0$ there exist $\epsilon>0$ and $n_0$ such that the following holds for all real numbers $x\ge n_0$.
Suppose that $\mathcal{H}$ is a cancellative $r$-graph with
\begin{align}\label{equ:cancellative-assumptions}
|\partial\mathcal{H}| = \frac{x^{r-1}}{r^{r-2}}
\quad{\rm and}\quad
|\mathcal{H}| \ge (1-\epsilon)\left(\frac{x}{r}\right)^{r}.
\end{align}
Then there exists a set $V'\subset V(\mathcal{H})$ of size at most $\lceil x\rceil$ such that $\mathcal{H}$ is a subgraph
of a complete $r$-partite $r$-graph on $V'$ after removing at most $\delta x^r$ edges.
\end{theorem}

\subsection{Expansion of cliques}\label{SUBSEC:cancellative-hypergraph}
For $\ell \ge r\ge 2$ let $\mathcal{K}_{\ell+1}^{r}$ denote the collection of all $r$-graphs $F$ with at most $\binom{{\ell}+1}{2}$
edges such that for some $({\ell}+1)$-set $S$ (called the core of $F$), every pair $\{u,v\}\subset S$ is covered by an edge in $F$.
Let the $r$-graph $H_{\ell+1}^{r}$ be obtained from the complete graph $K_{\ell+1}$ by adding $r-2$ new vertices to every edge.
The $r$-graph $H_{\ell+1}^{r}$ is called the {\em expansion} of $K_{\ell+1}$,
and it is easy to see that $H_{\ell+1}^{r} \in \mathcal{K}_{\ell+1}^{r}$.
Also note that $\{H_{\ell+1}^{2}\} = \mathcal{K}_{\ell+1}^{2} = \{K_{\ell+1}\}$.

The family $\mathcal{K}_{\ell+1}^{r}$ was introduced by Mubayi in \cite{MU06} as a way to extend Tur\'{a}n's theorem to hypergraphs.
He proved that $\textrm{ex}(n,\mathcal{K}_{\ell+1}^{r}) = t_{r}(n,\ell)$, and moreover, $T_{r}(n,\ell)$ is the unique
$\mathcal{K}_{\ell+1}^{r}$-free $r$-graph on $n$ vertices with exactly $t_{r}(n,\ell)$ edges.
Pikhurko~$\cite{PI13}$ improved this result by showing that
$\textrm{ex}(n,H_{\ell+1}^{r}) = t_{r}(n,\ell)$ for sufficiently large $n$
and $T_{r}(n,\ell)$ is also the unique $H_{\ell+1}^{r}$-free $r$-graph on $n$ vertices with exactly $t_{r}(n,\ell)$ edges.
One key tool used by Pikhurko is the following stability theorem,
which extends the Erd\H{o}s-Simonovits stability theorem for graphs \cite{SI68}.

\begin{theorem}[Stability, see \cite{MU06,PI13,LIU19}]\label{THM:old-stability-H}
Let $\ell \ge r\ge 2$ be integers.
For every $\delta > 0$ there exists $\epsilon > 0$ and $n_0 = n_0(\ell,r,\delta)$ such that
the following holds for all $n\ge n_0$.
Suppose that $\mathcal{H}$ is an $H_{\ell+1}^{r}$-free $r$-graph on $n$ vertices with
at least $(1-\epsilon)t_{r}(n,\ell)$ edges.
Then $\mathcal{H}$ is a subgraph of $T_{r}(n,\ell)$ after removing at most $\delta n^r$ edges.
\end{theorem}

In \cite{LM19A} the following Kruskal-Katona type result was proved for $\mathcal{K}_{\ell+1}^{r}$-free $r$-graphs.

\begin{theorem}[\cite{LM19A}]\label{THM:LM-clique-expansion}
Let $\ell \ge r\ge 2$ be integers and $\mathcal{H}$ be a $\mathcal{K}_{\ell+1}^{r}$-free $r$-graph.
Suppose that $|\partial\mathcal{H}| = \binom{\ell}{r-1}(x/\ell)^{r-1}$ for some real number $x\ge \ell$.
Then $|\mathcal{H}| \le \binom{\ell}{r}(x/\ell)^{r}$.
In other words,
$$|\mathcal{H}| \le \binom{\ell}{r}\left(\frac{|\partial\mathcal{H}|}{\binom{\ell}{r-1}}\right)^{r/(r-1)}.$$
\end{theorem}

Note that the inequality in Theorem~\ref{THM:LM-clique-expansion} is tight for the $r$-graph $T_{r}(n,m,\ell)$
if $m$ is a multiple of $\ell$.

We prove the following stability result for Theorem~\ref{THM:LM-clique-expansion}.

\begin{theorem}\label{THM:stability-expansion-clique-KK-type}
Let $\ell \ge r\ge 2$ be integers.
For $\delta>0$ there exist $\epsilon > 0$ and $n_0$ such that the following holds for all real numbers $x\ge n_0$.
Suppose that $\mathcal{H}$ is an $\mathcal{K}_{\ell+1}^{r}$-free $r$-graph with
\begin{align}\label{equ:expansion-assumption}
|\partial\mathcal{H}| = \binom{\ell}{r-1}\left(\frac{x}{\ell}\right)^{r-1}
\quad{\rm and}\quad
|\mathcal{H}| \ge (1-\epsilon)\binom{\ell}{r}\left(\frac{x}{\ell}\right)^{r}.
\end{align}
Then there exists a set $V'\subset V(\mathcal{H})$ of size at most $\lceil x \rceil$ such that
$\mathcal{H}$ is a subgraph of a complete $\ell$-partite $r$-graph on $V'$
after removing at most $\delta x^r$ edges.
\end{theorem}

This paper is organized as follows.
In Section~\ref{SEC:proof-KK-stability-cancellative}, we prove Theorem~\ref{THM:stability-cancellative-KK-type}.
Section~\ref{SEC:proof-KK-stability-clique-expansion}, we prove Theorem~\ref{THM:stability-expansion-clique-KK-type}.
In Section~\ref{SEC:remarks}, we include some concluding remarks.

\section{Cancellative hypergraphs}\label{SEC:proof-KK-stability-cancellative}
We prove Theorem~\ref{THM:stability-cancellative-KK-type} in this section.

\subsection{Preliminaries}\label{SUBSEC:prelim-cancellative}
For every integer $i\in[r-1]$ the {\em $i$-th shadow} $\partial_{i}\mathcal{H}$ of an $r$-graph $\mathcal{H}$ is
\begin{align}
\partial_{i}\mathcal{H} = \left\{ A\in \binom{V(\mathcal{H})}{r-i}\colon \exists B\in \mathcal{H} \text{ such that }A\subset B \right\}. \notag
\end{align}
The $(r-1)$-graph $\partial_{1}\mathcal{H}$ is also called the shadow of $\mathcal{H}$ and denoted by $\partial\mathcal{H}$.
For a vertex $v \in V(\mathcal{H})$ the {\em link} of $v$ in $\mathcal{H}$ is
\begin{align}
L_{\mathcal{H}}(v) = \left\{ A\in \partial\mathcal{H}\colon \{v\}\cup A \subset \mathcal{H} \right\}. \notag
\end{align}
The degree of $v$ is $d_{\mathcal{H}}(v) = |L_{\mathcal{H}}(v)|$.
For a set $S\subset V(\mathcal{H})$ of size at most $r-1$ the {\em neighborhood} of $S$ is
\begin{align}
N_{\mathcal{H}}(S) = \left\{ v\in V(\mathcal{H})\setminus S\colon \exists A\in \mathcal{H} \text{ such that } S\cup \{v\}\subset A \right\}. \notag
\end{align}
When $S = \{v\}$ for some $v\in V(\mathcal{H})$ we write $N_{\mathcal{H}}(v)$ instead of $N_{\mathcal{H}}(\{v\})$.
The {\em degree sum} of $S$ is defined as
\begin{align}
\sigma_{\mathcal{H}}(S) = \sum_{v\in S}d_{\mathcal{H}}(v). \notag
\end{align}
For every hypergraph $\mathcal{H}$ let
\begin{align}
\hat{\sigma}_{\mathcal{H}} = \max\left\{ \sigma_{\mathcal{H}}(E)\colon E \in \mathcal{H} \right\}. \notag
\end{align}
We will omit the subscript $\mathcal{H}$ from the notations above if it is clear from context.

We say a set $S$ of vertices is {\em $2$-covered} in $\mathcal{H}$
if every pair of vertices in $S$ is contained in some edge of $\mathcal{H}$,
or equivalent, if $S$ induces a complete graph in the graph $\partial_{r-2}\mathcal{H}$.

For cancellative hypergraphs we have the following lemma for $2$-covered sets.

\begin{lemma}[\cite{LM19A}]\label{LEMMA:cancellative-clique-in-r-2-shadow}
Let $r\ge 2$ be an integer and $\mathcal{H}$ be a cancellative $r$-graph.
Suppose that $S \subset V(\mathcal{H})$ is a $2$-covered set.
Then $L(v) \cap L(u) = \emptyset$ for every pair $\{u,v\} \subset S$.
In particular, $\sigma(S) = \sum_{v\in S}d(v) \le |\partial\mathcal{H}|$.
\end{lemma}

The following inequalities which can be found in {\cite[Section 4.1]{LM19A}} will be important for our proofs.

\begin{lemma}[\cite{LM19A}]\label{LEMMA:inequalities-cancellative-in-LM}
Let $r\ge 2$ be an integer, $\mathcal{H}$ be a cancellative $r$-graph, and $E\in \mathcal{H}$ be an edge with $\sigma(E) = \hat{\sigma}$.
Then
\begin{align}\label{inequ-Thm4.4-01}
|\mathcal{H}|
& \le \frac{|\partial \mathcal{H}|^{\frac{r-2}{r-1}}}{r(r-1)^{1/(r-1)}}
     \left(\left(|\partial \mathcal{H}|-\frac{\hat{\sigma}}{r}\right)\hat{\sigma} \right)^{\frac{1}{r-1}},
\end{align}
\begin{align}\label{inequ-Thm4.4-02}
|\mathcal{H}|
& \le \frac{|\partial \mathcal{H}|^{\frac{r-2}{r-1}}}{r(r-1)^{1/(r-1)}}
     \left( \sum_{v\in E}d(v)\left(\hat{\sigma}-d(v)\right) + \left(|\partial \mathcal{H}| - \hat{\sigma} \right)\hat{\sigma} \right)^{\frac{1}{r-1}},
\end{align}
\begin{align}\label{inequ-Thm4.4-03}
|\mathcal{H}|
& \le \frac{|\partial \mathcal{H}|^{\frac{r-2}{r-1}}}{r(r-1)^{1/(r-1)}}
      \left(\sum_{v \in E}\sum_{S \in L(v)}\sigma(S) + \sum_{S \in \partial\mathcal{H}\setminus \bigcup_{v\in E}L(v)}\sigma(S) \right)^{\frac{1}{r-1}},
\end{align}
and
\begin{align}\label{inequ-Thm4.4-04}
\frac{1}{r(r-1)} \sum_{v \in V(\mathcal{H})}\left(d(v)\right)^{\frac{1}{r-1}}|\partial L(v)|
\le \left(\frac{|\partial \mathcal{H}|}{r} \right)^{\frac{r}{r-1}}.
\end{align}
\end{lemma}

We will also use the following property of cancellative hypergraphs.

\begin{lemma}\label{Lemma:N(s)-cap-Uv=empty}
Let $r \ge 2$ be an integer and $\mathcal{H}$ be a cancellative $r$-graph.
Then for every $v\in V(\mathcal{H})$ and every $A \in L(v)$ we have $N(A) \cap N(v) = \emptyset$.
In other words, $N(A) \subset V(\mathcal{H})\setminus N(A)$.
\end{lemma}
\begin{proof}[Proof of Lemma~\ref{Lemma:N(s)-cap-Uv=empty}]
Suppose to the contrary that there exists $A\in L(v)$ and $u\in N(v)$ such that $u\in N(A)$.
Since $|A| = r-1$, $|A\cup \{u\}| = r$.
So, $A\cup \{u\}$ is an edge in $\mathcal{H}$.
On the other hand, since $A\in L(v)$, $A\cup \{v\}$ is also an edge in $\mathcal{H}$.
However, $u\in N(v)$ implies that there is an edge $B\in \mathcal{H}$ such that $\{u,v\} \subset B$,
which contradicts the assumption that $\mathcal{H}$ is cancellative.
\end{proof}

The following easy lemma will help us to simplify some calculations.

\begin{lemma}\label{LEMMA:function-concentrate}
Let $V$ be a finite set, $f\colon V \to \mathbb{R}$ be a map, and $\delta_1, \delta_2 > 0$ be two real numbers.
Let $\bar{f} = \left(\sum_{v\in V}f(v)\right)/|V|$ be the average value of $f$ on $V$, and
suppose that $\max_{v\in V}\{f(v)\} \le \bar{f} + \delta_2$.
Then the set  $V_{s} = \{v\in V\colon f(v) \le \bar{f}-\delta_1\}$ satisfies
\begin{align}
|V_{s}| \le \frac{\delta_2}{\delta_1+\delta_2} |V|. \notag
\end{align}
\end{lemma}
\begin{proof}[Proof of Lemma~\ref{LEMMA:function-concentrate}]
By assumption,
\begin{align}
|V|\bar{f}
= \sum_{v\in V}f(v)
= \sum_{v'\in V_{s}}f(v') + \sum_{v \in V\setminus V_{s}}f(v)
& \le |V_{s}| \left(\bar{f}-\delta_1\right) + \left(|V|-|V_{s}|\right)\left(\bar{f}+\delta_2\right) \notag\\
& = |V|\bar{f}+\delta_2|V| - \left(\delta_1+\delta_2\right)|V_{s}|, \notag
\end{align}
which implies that $|V_{s}| \le \delta_2|V|/(\delta_1+\delta_2)$.
\end{proof}

For two nonnegative numbers $x,y\in \mathbb{R}$ and $\epsilon\in [0,1]$
we write $x = (1\pm \epsilon)y$ if $x$ satisfies $(1-\epsilon)y \le x \le (1+\epsilon)y$.

\subsection{Proof of Theorem~\ref{THM:stability-cancellative-KK-type}}
We prove Theorem~\ref{THM:stability-cancellative-KK-type} in this section.
The most technical parts are contained in
proofs of Lemma~\ref{LEMMA:cancellative-KK-stability-reduction} and Lemma~\ref{LEMMA:cancellative-KK-stability-after-reduce}.
In Lemma~\ref{LEMMA:cancellative-KK-stability-reduction} we show that
the proof of Theorem~\ref{THM:stability-cancellative-KK-type} can be reduced to
the same problem with two extra assumptions that $|\partial\mathcal{H}|$ is close to
$|\partial T_{r}(\lceil x\rceil ,r)|$ and $|\mathcal{H}|$ is close to $|T_{r}(\lceil x\rceil,r)|$.
In Lemma~\ref{LEMMA:cancellative-KK-stability-after-reduce} we prove the stability result for
hypergraphs with these two extra assumptions.

\begin{lemma}\label{LEMMA:cancellative-KK-stability-reduction}
Let $r\ge 2$ be an integer, $\epsilon > 0$ be a sufficiently small constant, and $x>0$ be sufficiently large real number.
Suppose that $\mathcal{H}$ is a cancellative $r$-graph satisfying $(\ref{equ:cancellative-assumptions})$.
Then there exists a set $U \subset V(\mathcal{H})$ of size $(1\pm \epsilon_1)x$
such that the induced subgraph of $\mathcal{H}$ on $U$ satisfies
\begin{align}
|\partial (\mathcal{H}[U])|  \ge \left(1-\epsilon_1\right) \frac{x^{r-1}}{r^{r-2}}
\quad{\rm and}\quad
|\mathcal{H}[U]| \ge \left(1- \epsilon_1\right) \left(\frac{x}{r}\right)^{r}, \notag
\end{align}
where $\epsilon_1 = 35r^{4}\epsilon^{1/2}$.
\end{lemma}

\begin{lemma}\label{LEMMA:cancellative-KK-stability-after-reduce}
Let $r \ge 2$ be an integer.
For every $\delta > 0$ there exists $\epsilon>0$ and $n_0$ such that the following holds for all $n\ge n_0$.
Suppose that $\mathcal{H}$ is a cancellative $r$-graph on $n$ vertices with
\begin{align}\label{equ:assumption-cacnel-after-reduction}
|\partial\mathcal{H}| = (1\pm \epsilon)\frac{n^{r-1}}{r^{r-2}}
\quad{\rm and}\quad
|\mathcal{H}| \ge (1-\epsilon)\left(\frac{n}{r}\right)^{r}.
\end{align}
Then $\mathcal{H}$ is a subgraph of $T_{r}(n,r)$ after removing at most $\delta n^r$ edges.
\end{lemma}

First let us show that Lemmas~\ref{LEMMA:cancellative-KK-stability-reduction}
and~\ref{LEMMA:cancellative-KK-stability-after-reduce} imply Theorem~\ref{THM:stability-cancellative-KK-type}.

\begin{proof}[Proof of Theorem~\ref{THM:stability-cancellative-KK-type} using Lemmas~\ref{LEMMA:cancellative-KK-stability-reduction}
and~\ref{LEMMA:cancellative-KK-stability-after-reduce}]
Fix the integer $r\ge 2$ and the constant $\delta > 0$.
Let $\epsilon>0$ be a sufficiently small constant (whose value can be determined in the following proof)
and $x>0$ be a sufficiently large real number.
Let $\mathcal{H}$ be a cancellative $r$-graph satisfying $(\ref{equ:cancellative-assumptions})$.
By Lemma~\ref{LEMMA:cancellative-KK-stability-reduction} there exists a set $U \subset V(\mathcal{H})$
of size $(1\pm \epsilon_1)x$ such that the induced subgraph $\mathcal{H}$ satisfies
\begin{align}
\left(1-\epsilon_1\right) \frac{x^{r-1}}{r^{r-2}} \le |\partial (\mathcal{H}[U])| \le \frac{x^{r-1}}{r^{r-2}}
\quad{\rm and}\quad
|\mathcal{H}[U]| \ge \left(1- \epsilon_1\right)\left(\frac{x}{r}\right)^r, \notag
\end{align}
where $\epsilon_1 = 40r^{2r}\epsilon^{1/2}$.
Let $m = |U| = (1\pm \epsilon_1)x$.
Then $x = (1\pm 2\epsilon_1)m$ and the inequalities above imply that
\begin{align}
\left(1-4r\epsilon_1\right) \frac{m^{r-1}}{r^{r-2}} \le |\partial (\mathcal{H}[U])| \le (1+4r\epsilon_1) \frac{m^{r-1}}{r^{r-2}}
\quad{\rm and}\quad
|\mathcal{H}[U]| \ge \left(1- 4r\epsilon_1\right)\left(\frac{m}{r}\right)^r, \notag
\end{align}
Lemma~\ref{LEMMA:cancellative-KK-stability-after-reduce} applied to the $r$-graph $\mathcal{H}[U]$
shows that $\mathcal{H}[U]$ is a subgraph of $T_{r}(m,r)$ after removing at most $\delta_1 m^r\le 2\delta_1 x^r$ edges,
where $\delta_{1} = \delta_{1}(4r\epsilon_1)$ is the constant guaranteed by Lemma~\ref{LEMMA:cancellative-KK-stability-after-reduce}.
If $m \le \lceil x\rceil$, then let $V' = U$ and we are done.
Otherwise we replace $U$ by any $\lceil x\rceil$-subset $V'$ of it,
and since $m \le (1+\epsilon_1)x$, we only loss at most $\epsilon_1 x^r$ edges.
Therefore, we can remove at most $\epsilon_1\left(x/r\right)^r+2\delta_1 x^r + \epsilon_1 x^r$ edges from
$\mathcal{H}$ to obtain an $r$-partite $r$-graph on at most $\lceil x\rceil$ vertices.
\end{proof}

\subsection{Proof of Lemma~\ref{LEMMA:cancellative-KK-stability-reduction}}\label{SUBSEC:proof-lemma-cancellative-KK-stability-reduction}
We prove Lemma~\ref{LEMMA:cancellative-KK-stability-reduction} in this section.
Recall that for a hypergraph $\mathcal{H}$ and a set $S\subset V(\mathcal{H})$
\begin{align}
\sigma_{\mathcal{H}}(S) = \sum_{v\in S}d_{\mathcal{H}}(v),
\quad{\rm and}\quad
\hat{\sigma}_{\mathcal{H}} = \max\left\{ \sigma_{\mathcal{H}}(E)\colon E \in \mathcal{H} \right\}. \notag
\end{align}
The subscript $\mathcal{H}$ will be omitted in the following proof.

\begin{proof}[Proof of Lemma~\ref{LEMMA:cancellative-KK-stability-reduction}]
We may assume that $r\ge 3$ since the case $r = 2$ follows from the Erd\H{o}s-Simonovits stability theorem \cite{SI68} (for $K_3$-free graphs).
Let $\epsilon>0$ be sufficiently small, $x\ge 0$ be sufficiently large, and
$\mathcal{H}$ be a cancellative $r$-graph satisfying assumptions in Lemma~\ref{LEMMA:cancellative-KK-stability-reduction}.
Fix an edge $E\in \mathcal{H}$ with $\sigma(E) = \hat{\sigma}$.

\begin{claim}\label{claim-low-bound-sigma-hat}
We have $\left( 1- 2r\epsilon \right)|\partial \mathcal{H}| < \hat{\sigma} \le |\partial \mathcal{H}|$.
\end{claim}
\begin{proof}[Proof of Claim \ref{claim-low-bound-sigma-hat}]
The inequality $\hat{\sigma} \le |\partial \mathcal{H}|$ follows from Lemma~\ref{LEMMA:cancellative-clique-in-r-2-shadow},
so we may focus on the lower bound for $\hat{\sigma}$.
It follows from $(\ref{equ:cancellative-assumptions})$ and (\ref{inequ-Thm4.4-01}) that
\begin{align}
(1-\epsilon)\left(\frac{|\partial \mathcal{H}|}{r}\right)^{\frac{r}{r-1}}
\le |\mathcal{H}|
\le \frac{|\partial \mathcal{H}|^{\frac{r-2}{r-1}}}{r(r-1)^{1/(r-1)}}
     \left(\left(|\partial \mathcal{H}|-\frac{\hat{\sigma}}{r}\right)\hat{\sigma} \right)^{\frac{1}{r-1}}. \notag
\end{align}
So,
\begin{align}
\left(|\partial \mathcal{H}|-\frac{\hat{\sigma}}{r}\right)\hat{\sigma}
\ge (1-\epsilon)^{r-1} \frac{r-1}{r}|\partial \mathcal{H}|^2
& \ge (1-(r-1)\epsilon)\frac{r-1}{r} |\partial \mathcal{H}|^2 \notag\\
& =  \frac{r-1}{r}|\partial\mathcal{H}|^2 - \epsilon \frac{(r-1)^2}{r}|\partial \mathcal{H}|^2. \notag
\end{align}
Solving this quadratic inequality we obtain that $\hat{\sigma} \le \left(1- 2r\epsilon\right)|\partial \mathcal{H}|$
(the other solution is greater than $|\partial\mathcal{H}|$, which is not possible).
\end{proof}

\begin{claim}\label{claim-low-bound-d(v)}
We have $\left| d(v) - {\hat{\sigma}}/{r} \right| < 2r\epsilon^{1/2}\hat{\sigma}$ for every vertex $v \in E$.
\end{claim}
\begin{proof}[Proof of Claim \ref{claim-low-bound-d(v)}]
First, we prove that
\begin{align}\label{inequ-claim-d(v)[sigma-d(v)]-low-bound}
\sum_{v\in E}d(v)\left(\hat{\sigma}-d(v)\right) > \left( \frac{r-1}{r} - 2r\epsilon \right) \hat{\sigma}^{2}.
\end{align}
Suppose that (\ref{inequ-claim-d(v)[sigma-d(v)]-low-bound}) is not true.
Then by $(\ref{inequ-Thm4.4-02})$,
\begin{align}
|\mathcal{H}|
& \le \frac{|\partial \mathcal{H}|^{\frac{r-2}{r-1}}}{r(r-1)^{1/(r-1)}}
        \left( \sum_{v\in E}d(v)\left(\hat{\sigma}-d(v)\right)
         + \left(|\partial \mathcal{H}| - \hat{\sigma} \right)\hat{\sigma} \right)^{1/(r-1)} \notag \\
& \le \frac{|\partial \mathcal{H}|^{\frac{r-2}{r-1}}}{r(r-1)^{1/(r-1)}}
     \left( \left( \frac{r-1}{r} -2r\epsilon \right) \hat{\sigma}^{2}
      + \left(|\partial \mathcal{H}| - \hat{\sigma} \right)\hat{\sigma} \right)^{1/(r-1)} \notag\\
& \le \frac{|\partial \mathcal{H}|^{\frac{r-2}{r-1}}}{r(r-1)^{1/(r-1)}}
     \left(\left(|\partial \mathcal{H}| - \left(\frac{1}{r}+2r\epsilon\right)\hat{\sigma} \right)\hat{\sigma} \right)^{1/(r-1)} \notag
\end{align}
It follows from $\hat{\sigma} \le |\partial\mathcal{H}|$ (Claim~\ref{claim-low-bound-sigma-hat}) that
\begin{align}
|\mathcal{H}|
& \le \frac{|\partial \mathcal{H}|^{\frac{r-2}{r-1}}}{r(r-1)^{1/(r-1)}}
     \left(\left(\frac{r-1}{r} - 2r\epsilon \right)|\partial \mathcal{H}|^2 \right)^{\frac{1}{r-1}} \notag \\
& = \left(1-\frac{2r^2}{r-1}\epsilon\right)^{\frac{1}{r-1}} \left( \frac{|\partial\mathcal{H}|}{r} \right)^{\frac{r}{r-1}}
< (1-\epsilon)\left( \frac{|\partial\mathcal{H}|}{r} \right)^{\frac{r}{r-1}}, \notag
\end{align}
a contradiction. Therefore, (\ref{inequ-claim-d(v)[sigma-d(v)]-low-bound}) is true.

Now suppose that Claim~\ref{claim-low-bound-d(v)} is not true.
Assume that $E = \{v_1,\ldots, v_r\}$ and without loss of generality we may assume that
$\left| d(v_1) - {\hat{\sigma}}/{r} \right| \ge 2r\epsilon^{1/2}\hat{\sigma}$.
Then by Jensen's inequality
\begin{align}
\sum_{i \in [r]}d(v_i)\left(\hat{\sigma}-d(v_i)\right)
& = d(v_1)\left(\hat{\sigma}-d(v_1)\right) +  \sum_{i = 2}^{r} d(v_i)\left(\hat{\sigma}-d(v_i)\right) \notag \\
& \le d(v_1)\left(\hat{\sigma}-d(v_1)\right)
  + \left( \sum_{i = 2}^{r} d(v_i) \right) \left( \hat{\sigma} - \frac{\sum_{i = 2}^{r} d(v_i)}{r-1} \right) \notag \\
& = d(v_1)\left(\hat{\sigma}-d(v_1)\right)
  + \left( \hat{\sigma} - d(v_1) \right) \left( \hat{\sigma} - \frac{\hat{\sigma} - d(v_1)}{r-1} \right) \notag \\
& = \frac{r-2}{r-1} \left(\hat{\sigma}-d(v_1)\right)  \left( \hat{\sigma} + \frac{r}{r-2} d(v_1) \right) \notag \\
& = \frac{r-1}{r}\hat{\sigma}^2 - \frac{r}{r-1}\left(d(v_1)-\frac{\hat{\sigma}}{r}\right)^2
< \frac{r-1}{r}\hat{\sigma}^2 - 2r\epsilon \hat{\sigma}^2, \notag
\end{align}
which contradicts (\ref{inequ-claim-d(v)[sigma-d(v)]-low-bound}).
\end{proof}

For every $v \in E$ let
$$\mathcal{L}_v = \left\{ S \in L(v)\colon \sigma(S) \ge \left(1-\epsilon^{1/2}\right)\left(\hat{\sigma}-d(v)\right) \right\}.$$

\begin{claim}\label{claim-mathcal(L)(v)-lower-bound}
We have $|\mathcal{L}_v| \ge (1-4r^2 \epsilon^{1/2})d(v)$ for every $v \in E$.
\end{claim}
\begin{proof}[Proof of Claim \ref{claim-mathcal(L)(v)-lower-bound}]
First we show that for every $v \in E$
\begin{align}\label{inequ-sum-sigma(S)-L_v}
\sum_{S \in L(v)} \sigma(S) \ge (1- 4r^2 \epsilon) d(v) \left( \hat{\sigma} - d(v) \right).
\end{align}
Suppose that $(\ref{inequ-sum-sigma(S)-L_v})$ is not true and let $u \in E$ be a counterexample.
Then
\begin{align}
\sum_{v \in E}\sum_{S \in L(v)}\sigma(S)
& = \sum_{S \in L(u)} \sigma(S) + \sum_{v \in E\setminus \{u\}}\sum_{S \in L(v)}\sigma(S) \notag \\
& \le (1-4r^2 \epsilon) d(u) \left( \hat{\sigma} - d(u) \right) + \sum_{v \in E\setminus \{u\}}d(v)\left( \hat{\sigma} - d(v) \right) \notag \\
& \le (1-2r\epsilon) \sum_{v \in E}d(v) \left( \hat{\sigma} - d(v) \right) \notag\\
& \quad + 2r\epsilon \sum_{v \in E}d(v) \left( \hat{\sigma} - d(v) \right) - 4r^2\epsilon d(u) \left( \hat{\sigma} - d(u) \right). \notag
\end{align}
Due to Claim~\ref{claim-low-bound-d(v)}, it is easy to see that
$\sum_{v \in E}d(v) \left( \hat{\sigma} - d(v) \right) < 2r d(u) \left( \hat{\sigma} - d(u)\right)$.
Therefore, by Jensen's inequality,
\begin{align}
\sum_{v \in E}\sum_{S \in L(v)}\sigma(S)
& \le (1-2r\epsilon) \sum_{v \in E}d(v) \left( \hat{\sigma} - d(v) \right) \notag\\
& \le (1-2r\epsilon)\left(\sum_{v \in E}d(v)\right)\left(\hat{\sigma}-\frac{\sum_{v \in E}d(v)}{r}\right)
 = (1-2r\epsilon) \frac{r-1}{r}\hat{\sigma}^{2}. \notag
\end{align}
Then it follows from $(\ref{inequ-Thm4.4-03})$ that
\begin{align}
|\mathcal{H}|
&\le \frac{1}{r(r-1)^{\frac{1}{r-1}}} |\partial \mathcal{H}|^{\frac{r-2}{r-1}}
      \left(\sum_{v \in E}\sum_{S \in L(v)}\sigma(S) + \sum_{S \in \partial\mathcal{H}\setminus \bigcup_{v\in E}L(v)}\sigma(S) \right)^{\frac{1}{r-1}}\notag\\
& \le \frac{1}{r(r-1)^{\frac{1}{r-1}}} |\partial \mathcal{H}|^{\frac{r-2}{r-1}}
    \left((1-2r\epsilon) \frac{r-1}{r} \hat{\sigma}^{2} + \left(|\partial\mathcal{H}|-\hat{\sigma}\right)\hat{\sigma} \right)^{\frac{1}{r-1}}\notag\\
& \le \frac{1}{r(r-1)^{\frac{1}{r-1}}} |\partial \mathcal{H}|^{\frac{r-2}{r-1}}
    \left(\left(|\partial\mathcal{H}|-\left(\frac{1}{r}+2(r-1)\epsilon \right)\hat{\sigma}\right)\hat{\sigma} \right)^{\frac{1}{r-1}}. \notag
\end{align}
Then, it follows from $\hat{\sigma} \le |\partial\mathcal{H}|$ (Claim~\ref{claim-low-bound-sigma-hat}) that
\begin{align}
|\mathcal{H}|
\le \left(1-2r\epsilon\right)^{\frac{1}{r-1}} \left( \frac{|\partial\mathcal{H}|}{r} \right)^{\frac{r}{r-1}}
< (1-\epsilon)\left( \frac{|\partial\mathcal{H}|}{r} \right)^{\frac{r}{r-1}}, \notag
\end{align}
a contradiction.
Therefore, $(\ref{inequ-sum-sigma(S)-L_v})$ holds for every $v \in E$.
Then, apply Lemma~\ref{LEMMA:function-concentrate} with $V = L(v)$ and $f(A) = \sigma(A)$ for every $A\in L(v)$ we obtain
\begin{align}
|L(v) \setminus \mathcal{L}_v|
& \le \frac{\left(\hat{\sigma} - d(v)\right) - \sum_{S\in L(v)}\sigma(S)/d(v)}
            {\left(\hat{\sigma} - d(v)\right) - \left( 1- \epsilon^{1/2} \right)\left( \hat{\sigma} - d(v) \right)} \cdot d(v) \notag\\
& \le \frac{\left(\hat{\sigma} - d(v)\right) - (1- 4r^2 \epsilon)\left( \hat{\sigma} - d(v) \right)}
             {\epsilon^{1/2}\left( \hat{\sigma} - d(v) \right)} \cdot d(v)
 \le 4r^2 \epsilon^{1/2} d(v). \notag
\end{align}
The completes the proof of Claim~\ref{claim-mathcal(L)(v)-lower-bound}.
\end{proof}

Let $$\mathcal{G} = \left\{S \in \partial\mathcal{H}\colon
                        \sigma(S)\ge \left(\frac{r-1}{r}-2r\epsilon^{1/2}\right)|\partial \mathcal{H}| \right\}.$$

\begin{claim}\label{claim-mathcal(G)-lower-bound}
We have $|\mathcal{G}| \ge (1-8r^2 \epsilon^{1/2}) |\partial\mathcal{H}|$.
\end{claim}
\begin{proof}[Proof of Claim \ref{claim-mathcal(G)-lower-bound}]
By Claims~\ref{claim-low-bound-d(v)} and~\ref{claim-low-bound-sigma-hat},
for every $v \in E$ and $S \in \mathcal{L}_{v}$ we have
\begin{align}
\sigma(S)
 \ge \left(1-\epsilon^{1/2}\right)\left(\hat{\sigma}-d(v)\right)
& \ge \left(1-\epsilon^{1/2}\right)\left(\frac{r-1}{r}-2r\epsilon^{1/2}\right)\hat{\sigma} \notag\\
& \ge \left(1-\epsilon^{1/2}\right)\left(\frac{r-1}{r}-2r\epsilon^{1/2}\right) \left(1-2r\epsilon\right) |\partial \mathcal{H}| \notag \\
& \ge \left(\frac{r-1}{r}-2r\epsilon^{1/2}\right)|\partial \mathcal{H}|. \notag
\end{align}
On the other hand, by Claims~\ref{claim-mathcal(L)(v)-lower-bound} and~\ref{claim-low-bound-sigma-hat},
\begin{align}
\sum_{v \in E}|\mathcal{L}_v|
\ge \sum_{v \in E}(1-4r^2 \epsilon^{1/2})d(v)
= (1-4r^2 \epsilon^{1/2}) \hat{\sigma}
\ge (1-8r^2 \epsilon^{1/2}) |\partial\mathcal{H}|. \notag
\end{align}
Therefore, by Lemma~\ref{LEMMA:cancellative-clique-in-r-2-shadow},
$|\mathcal{G}| \ge \sum_{v \in E}|\mathcal{L}_v| \ge (1-8r^2 \epsilon^{1/2}) |\partial\mathcal{H}|$.
\end{proof}

\begin{claim}\label{claim-Delta(H)-up-bound}
We have $\Delta(\mathcal{H}) \le \left( \frac{1}{r} + 3r\epsilon^{1/2}\right)|\partial\mathcal{H}|$.
\end{claim}
\begin{proof}[Proof of Claim \ref{claim-Delta(H)-up-bound}]
Suppose to the contrary that there exists a vertex $u \in V(\mathcal{H})$ with
\begin{align}
d(u) > \left( \frac{1}{r} + 3r\epsilon^{1/2} \right)|\partial\mathcal{H}|. \notag
\end{align}
Then, for every $S \in L(u)$,
\begin{align}
\sigma(S)
\le \hat{\sigma} - d(u)
< |\partial \mathcal{H}| -\left( \frac{1}{r} + 3r\epsilon^{1/2} \right)|\partial\mathcal{H}|
=  \left( \frac{r-1}{r} - 3r\epsilon^{1/2} \right)|\partial\mathcal{H}|. \notag
\end{align}
Therefore, $L(u) \cap \mathcal{G} = \emptyset$, and hence
\begin{align}
|\mathcal{G}|
 \le |\partial \mathcal{H}| - |d(u)|  < \frac{r-1}{r}|\partial\mathcal{H}|<
 (1-8r^2 \epsilon^{1/2}) |\partial\mathcal{H}|, \notag
\end{align}
which contradicts Claim \ref{claim-mathcal(G)-lower-bound}.
\end{proof}

Let $U = \partial_{r-2} \mathcal{G} \subset V(\mathcal{H})$.

\begin{claim}\label{claim-U-up-bound}
We have $|U| \le \left(1+6r^{3}\epsilon^{1/2}\right) r^{\frac{r-2}{r-1}}|\partial\mathcal{H}|^{\frac{1}{r-1}}$.
\end{claim}
\begin{proof}[Proof of Claim \ref{claim-U-up-bound}]
First we show that for every $v \in U$,
\begin{align}\label{inequ-d(v)-up-bound}
d(v) \ge \left(\frac{1}{r}-3r^{2}\epsilon^{1/2} \right)|\partial\mathcal{H}|.
\end{align}
Suppose that $(\ref{inequ-d(v)-up-bound})$ is not true and let $u \in U$ be a counterexample.
Then choose a set $S \in \mathcal{G}$ such that $u \in S$.
By the definition of $\mathcal{G}$,
\begin{align}
\sigma(S) \ge \left(\frac{r-1}{r}-2r\epsilon^{1/2}\right)|\partial \mathcal{H}|,\notag
\end{align}
so by the Pigeonhole principle, there exists $u' \in S\setminus\{u\}$ such that
\begin{align}
d(u')
\ge \frac{\sigma(S)-d(u)}{r-2}
& > \frac{\left((r-1)/r-2r\epsilon^{1/2} \right)|\partial \mathcal{H}|
   - \left(1/r-3r^{2}\epsilon^{1/2} \right)|\partial\mathcal{H}|}{r-2} \notag \\
& > \left( \frac{1}{r} + 3r\epsilon^{1/2} \right)|\partial\mathcal{H}|, \notag
\end{align}
which contradicts Claim~\ref{claim-Delta(H)-up-bound}.
Therefore, $(\ref{inequ-d(v)-up-bound})$ holds for every $v \in U$,
and it follows from $\sum_{v\in U}d(v) \le r|\mathcal{H}|$ and Theorem~\ref{THM:LM-cancellative} that
\begin{align}
|U|
\le \frac{r|\mathcal{H}|}{\left(1/r-3r^{2}\epsilon^{1/2} \right)|\partial\mathcal{H}|}
\le \frac{r\left({|\partial \mathcal{H}|}/{r} \right)^{\frac{r}{r-1}}}{\left(1/r-3r^{2}\epsilon^{1/2} \right)|\partial\mathcal{H}|}
 <  \left(1+6r^{3}\epsilon^{1/2}\right) r^{\frac{r-2}{r-1}}|\partial\mathcal{H}|^{\frac{1}{r-1}}. \notag
\end{align}
\end{proof}

\begin{claim}\label{claim-hat(mathcal(G))-low-bound}
We have $|\mathcal{H}[U]| \ge \left(1-33r^{4}\epsilon^{1/2}\right)\left(\frac{|\partial\mathcal{H}|}{r}\right)^{\frac{r}{r-1}}$.
\end{claim}
\begin{proof}[Proof of Claim \ref{claim-hat(mathcal(G))-low-bound}]
By $(\ref{inequ-d(v)-up-bound})$ and Claim~\ref{claim-mathcal(G)-lower-bound}, for every $u \in U$ we have
\begin{align}
d_{\mathcal{H}[U]}(u)
  \ge d_{\mathcal{H}}(u) - |\partial\mathcal{H}\setminus\mathcal{G}|
& \ge \left(\frac{1}{r}-3r^{2}\epsilon^{1/2} \right)|\partial\mathcal{H}| - 8r^2\epsilon^{1/2}|\partial\mathcal{H}| \notag\\
& = \left(\frac{1}{r}-11r^{2}\epsilon^{1/2} \right)|\partial\mathcal{H}|. \notag
\end{align}
For every $0 \le i \le r$ let $\mathcal{E}_{i}$ be the set of edges in $\mathcal{H}$ that have exactly $i$ vertices in $U$
and note that $\mathcal{E}_r = \mathcal{H}[U]$.
Then by Claim~\ref{claim-Delta(H)-up-bound} we have
\begin{align}
\sum_{i\in[r-1]}i|\mathcal{E}_i|
= \sum_{u\in U}d_{\mathcal{H}}(u) - r|\mathcal{E}_r|
& = \sum_{u\in U}d_{\mathcal{H}}(u) - \sum_{u\in U}d_{\mathcal{H}[U]}(u) \notag\\
& \le \sum_{u\in U}\left(\Delta(\mathcal{H}) - d_{\mathcal{H}[U]}(u)\right) \notag\\
& \le \left(\left( \frac{1}{r} + 3r\epsilon^{1/2}\right)|\partial\mathcal{H}|
        - \left(\frac{1}{r}-11r^{2}\epsilon^{1/2} \right)|\partial\mathcal{H}|\right)|U| \notag\\
& \le 12r^{2}\epsilon^{1/2}|\partial\mathcal{H}||U|. \notag
\end{align}
It follows from Claim~\ref{claim-U-up-bound} that
\begin{align}
\sum_{i\in[r-1]}i|\mathcal{E}_i|
\le 12r^{2}\epsilon^{1/2}|\partial\mathcal{H}| \cdot
        \left(1+6r^2\epsilon^{1/2}\right) r^{\frac{r-2}{r-1}}|\partial\mathcal{H}|^{\frac{1}{r-1}}
\le 24r^{2}\epsilon^{1/2}|\partial\mathcal{H}|^{\frac{r}{r-1}}. \notag
\end{align}
On the other hand, by Theorem~\ref{THM:LM-cancellative},
$|\mathcal{E}_0| \le \left(|\partial\mathcal{H}|-|\mathcal{G}|\right)^{\frac{r}{r-1}}
\le 8r^2 \epsilon^{1/2} |\partial\mathcal{H}|^{\frac{r}{r-1}}$.
Therefore,
\begin{align}
|\mathcal{H}[U]|
= |\mathcal{H}| - \sum_{i=0}^{r-1}|\mathcal{E}_{i}|
& \ge (1-\epsilon)\left(\frac{|\partial\mathcal{H}|}{r}\right)^{\frac{r}{r-1}}
    - 24r^{2}\epsilon^{1/2}|\partial\mathcal{H}|^{\frac{r}{r-1}}
    - 8r^2 \epsilon^{1/2} |\partial\mathcal{H}|^{\frac{r}{r-1}} \notag\\
& \ge \left(1-33r^{4}\epsilon^{1/2}\right)\left(\frac{|\partial\mathcal{H}|}{r}\right)^{\frac{r}{r-1}}. \notag
\end{align}
\end{proof}

\begin{claim}\label{claim-U-low-bound}
We have $|U| \ge \left(1-35r^{4}\epsilon^{1/2}\right) r^{\frac{r-2}{r-1}}|\partial\mathcal{H}|^{\frac{1}{r-1}}$.
\end{claim}
\begin{proof}[Proof of Claim \ref{claim-U-low-bound}]
It follows from Claims \ref{claim-Delta(H)-up-bound}, \ref{claim-hat(mathcal(G))-low-bound},
and $\sum_{u \in U}d_{\mathcal{H}}(u) \ge r|\mathcal{H}[U]|$ that
\begin{align}
|U|
\ge  \frac{r|\mathcal{H}[U]|}{\Delta(\mathcal{H})}
\ge \frac{r\left(1-33r^{4}\epsilon^{1/2}\right)\left(|\mathcal{H}|/r\right)^{\frac{r}{r-1}}}
     {\left(1/r+ 3r\epsilon^{1/2} \right)|\partial\mathcal{H}|}
\ge \left(1-35r^{4}\epsilon^{1/2}\right) r^{\frac{r-2}{r-1}}|\partial\mathcal{H}|^{\frac{1}{r-1}}. \notag
\end{align}
\end{proof}

Now Claims~\ref{claim-U-up-bound} and~\ref{claim-U-low-bound} and $|\partial\mathcal{H}| = x^{r-1}/r^{r-2}$
imply that $|U| = (1\pm \epsilon_1)x$.
Claim~\ref{claim-hat(mathcal(G))-low-bound} shows that
$$|\mathcal{H}[U]|
\ge \left(1-33r^{4}\epsilon^{1/2}\right)\left(\frac{|\partial\mathcal{H}|}{r}\right)^{\frac{r}{r-1}}
= \left(1-33r^{4}\epsilon^{1/2}\right)\left(\frac{x}{r}\right)^{r}
\ge (1-\epsilon_1)\left(\frac{x}{r}\right)^{r}.$$
Together with Theorem~\ref{THM:LM-cancellative} we obtain
\begin{align}
|\partial\left(\mathcal{H}[U]\right)|
\ge r|\mathcal{H}[U]|^{\frac{r-1}{r}}
\ge (1-\epsilon_1)\frac{x^{r-1}}{r^{r-2}}. \notag
\end{align}
\end{proof}

\subsection{Proof of Lemma~\ref{LEMMA:cancellative-KK-stability-after-reduce}}\label{SUBSEC:proof-lemma-cancellative-KK-stability-after-reduce}
\begin{proof}[Proof of Lemma~\ref{LEMMA:cancellative-KK-stability-after-reduce}]
The proof if by induction on $r$.
The case $r = 2$ follows from the Erd\H{o}s-Simonovits stability theorem \cite{SI68} (for $K_3$-free graphs).
So we may assume that $r \ge 3$.
Fix $r\ge 3$ and $\delta>0$.
Let $\epsilon>0$ be sufficiently small, $x>0$ be sufficiently large, and
$\mathcal{H}$ be a cancellative $r$-graph satisfying assumptions in Lemma~\ref{LEMMA:cancellative-KK-stability-after-reduce}.
Let
\begin{align}
V_{L} & = \left\{ v\in V(\mathcal{H})\colon
            d(v) \ge (1-\epsilon^{1/2})\left(\frac{|\partial L(v)|}{r-1} \right)^{\frac{r-1}{r-2}} \right\}, \notag\\
\widehat{V}_{L} & = \left\{v\in V(\mathcal{H})\colon
            d(v) \ge \left( \frac{1}{r} - 3r^{2}\epsilon^{1/2} \right) |\partial\mathcal{H}| \right\}, \notag
\end{align}
$V_{S} = V(\mathcal{H}) \setminus V_{L}$, and $\widehat{V}_{S} = V(\mathcal{H})\setminus \widehat{V}_{L}$.
It follows from the definition that for every $v \in V_S$,
\begin{align}\label{inequ-partial-Lv-low-bound}
|\partial L(v)| \ge \frac{(r-1)\left(d(v) \right)^{\frac{r-2}{r-1}}}{(1-\epsilon^{1/2})^{\frac{r-2}{r-1}}}.
\end{align}

\begin{claim}\label{claim-low-bound-hat-VL}
We have $|\widehat{V}_{L}| \ge \left(1 - 36r^{4}\epsilon^{1/2}\right)n$, and hence $|\widehat{V}_{S}|\le 36r^{4}\epsilon^{1/2}n$.
\end{claim}
\begin{proof}[Proof of Claim \ref{claim-low-bound-hat-VL}]
Since $|\partial\mathcal{H}| = (1\pm \epsilon)n^{r-1}/r^{r-2}$ and $|\mathcal{H}| \ge (1-\epsilon)(n/r)^{r}$,
it follows from Claim~\ref{claim-U-low-bound} and $(\ref{inequ-d(v)-up-bound})$
that there exists a set $U\subset V(\mathcal{H})$ of size at least
\begin{align}
\left(1-35r^{4}\epsilon^{1/2}\right) r^{\frac{r-2}{r-1}}|\partial\mathcal{H}|^{\frac{1}{r-1}}
\ge \left(1-35r^{4}\epsilon^{1/2}\right)\left(1-\epsilon\right)n
\ge \left(1-36r^{4}\epsilon^{1/2}\right)n, \notag
\end{align}
such that $d(v)\ge (1/r-3r^2 \epsilon^{1/2})|\partial\mathcal{H}|$ for every $v \in U$.
Therefore, $|\widehat{V}_{L}| \ge |U| \ge \left(1-36r^{4}\epsilon^{1/2}\right)n$,
and hence $|\widehat{V}_{S}| = n-|\widehat{V}_{L}| \le 36r^{4}\epsilon^{1/2}n$.
\end{proof}

\begin{claim}\label{claim-low-bound-VL}
We have $|V_{L}| \ge (1-37r^{4}\epsilon^{1/2})n$.
\end{claim}
\begin{proof}[Proof of Claim \ref{claim-low-bound-VL}]
It is easy to see that for every $v \in V(\mathcal{H})$
the $(r-1)$-graph $L(v)$ is also cancellative, so by Theorem~\ref{THM:LM-cancellative},
$d(v) \le \left({|\partial L(v)|}/(r-1)\right)^{(r-1)/(r-2)}$.
Therefore, by the definition of $V_{L}$,
\begin{align}
|\mathcal{H}|
& = \frac{1}{r}\sum_{v \in V(\mathcal{H})}d(v) \notag \\
& = \frac{1}{r} \left( \sum_{v \in V_{L}}\left(d(v)\right)^{\frac{1}{r-1}} \left(d(v)\right)^{\frac{r-2}{r-1}}
    +  \sum_{v \in V_{S}}\left(d(v)\right)^{\frac{1}{r-1}} \left(d(v)\right)^{\frac{r-2}{r-1}} \right) \notag\\
& \le \frac{1}{r(r-1)}\left( \sum_{v \in V_{L}}\left(d(v)\right)^{\frac{1}{r-1}}|\partial L(v)|
    + (1-\epsilon^{1/2})^{\frac{r-2}{r-1}}\sum_{v \in V_{S}}\left(d(v)\right)^{\frac{1}{r-1}}|\partial L(v)| \right)  \notag\\
& = \frac{1}{r(r-1)} \sum_{v \in V(\mathcal{H})}\left(d(v)\right)^{\frac{1}{r-1}}|\partial L(v)|
  - \frac{1-(1-\epsilon^{1/2})^{\frac{r-2}{r-1}}}{r(r-1)} \sum_{v \in V_{S}}\left(d(v)\right)^{\frac{1}{r-1}}|\partial L(v)|.\notag
\end{align}
Together with $(\ref{inequ-partial-Lv-low-bound})$ we obtain
\begin{align}
|\mathcal{H}|
& \le \frac{1}{r(r-1)} \sum_{v \in V(\mathcal{H})}\left(d(v)\right)^{\frac{1}{r-1}}|\partial L(v)|
  - \frac{1-(1-\epsilon^{1/2})^{\frac{r-2}{r-1}}}{(1-\epsilon^{1/2})^{\frac{r-2}{r-1}}} \frac{1}{r} \sum_{v \in V_{S}}d(v) \notag \\
& \le \frac{1}{r(r-1)} \sum_{v \in V(\mathcal{H})}\left(d(v)\right)^{\frac{1}{r-1}}|\partial L(v)|
  - \frac{\epsilon^{1/2}}{2r} \sum_{v \in V_{S}}d(v). \notag
\end{align}
So by $(\ref{inequ-Thm4.4-04})$,
\begin{align}
|\mathcal{H}|
\le  \left(\frac{|\partial \mathcal{H}|}{r} \right)^{\frac{r}{r-1}}
        - \frac{\epsilon^{1/2}}{2r} \sum_{v \in V_{S}}d(v). \notag
\end{align}
Then, it follows from $|\mathcal{H}| > (1-\epsilon)\left(|\partial\mathcal{H}|/r\right)^{r/(r-1)}
\ge \left({|\partial \mathcal{H}|}/{r}\right)^{\frac{r}{r-1}} - \epsilon |\partial \mathcal{H}|^{\frac{r}{r-1}}$
and Claim~\ref{claim-low-bound-hat-VL} that
\begin{align}
\left(\frac{|\partial \mathcal{H}|}{r} \right)^{\frac{r}{r-1}} - \epsilon |\partial \mathcal{H}|^{\frac{r}{r-1}}
& \le \left(\frac{|\partial \mathcal{H}|}{r} \right)^{\frac{r}{r-1}}
  - \frac{\epsilon^{1/2}}{2r} \sum_{v \in V_{S}}d(v) \notag \\
& \le \left(\frac{|\partial \mathcal{H}|}{r} \right)^{\frac{r}{r-1}}
  -  \frac{\epsilon^{1/2}}{2r} |V_{S}\setminus \widehat{V}_{S}|\left(\frac{1}{r}-3r^2\epsilon^{1/2}\right)|\partial\mathcal{H}| \notag\\
& \le \left(\frac{|\partial \mathcal{H}|}{r} \right)^{\frac{r}{r-1}}
   -  \frac{\epsilon^{1/2}}{3r^2} \left(|V_{S}|-36r^{4}\epsilon^{1/2}\right) |\partial\mathcal{H}|, \notag
\end{align}
which implies that $|V_{S}| \le 37r^{4}\epsilon^{1/2}n$.
Therefore, $|V_{L}| = n -|V_S| \ge \left(1- 37r^{4}\epsilon^{1/2}\right)n$.
\end{proof}

Claims \ref{claim-low-bound-hat-VL} and \ref{claim-low-bound-VL} imply that
$|V_{L} \cap \widehat{V}_{L}| > \left(1- 73r^2 \epsilon^{1/2}\right)n$.
Then due to $|\mathcal{H}| \ge (1-\epsilon)(n/r)^{r}$,
there exists an edge $\widehat{E} \in \mathcal{H}[V_{L} \cap \widehat{V}_{L}]$.
By the definition of $V_{L}$ and $\widehat{V}_L$, for every $v \in \widehat{E}$ we have
\begin{align}\label{inequ-d(v)-low-bound-relative}
d(v) \ge (1-\epsilon^{1/2})\left(\frac{|\partial L(v)|}{r-1} \right)^{\frac{r-1}{r-2}},
\end{align}
and
\begin{align}\label{inequ-d(v)-low-bound-absolute}
d(v)
\ge \left( \frac{1}{r} - 3r^2\epsilon^{1/2} \right) |\partial\mathcal{H}|
\ge \left( \frac{1}{r} - 3r^2\epsilon^{1/2} \right) (1-\epsilon)\frac{n^{r-1}}{r^{r-2}}
\ge \left(1-4r^3\epsilon^{1/2}\right)\frac{n^{r-1}}{r^{r-1}}.
\end{align}
On the other hand, since $\sum_{v \in \widehat{E}}d(v) \le |\partial \mathcal{H}|$,
(\ref{inequ-d(v)-low-bound-absolute}) implies that for every $v \in \widehat{E}$,
\begin{align}\label{inequ-d(v)-up-bound-absolute}
d(v)
\le |\partial \mathcal{H}| - (r-1)\left(1-4r^3\epsilon^{1/2}\right)\frac{n^{r-1}}{r^{r-1}}
\le \left(1+4r^4\epsilon^{1/2}\right)\frac{n^{r-1}}{r^{r-1}}.
\end{align}
Notice that $L(v)$ is a cancellative $(r-1)$-graph.
So $(\ref{inequ-d(v)-low-bound-relative})$ and Theorem~\ref{THM:LM-cancellative} imply that
\begin{align}\label{inequ-mathcal(H)v-low-bound-relative}
(1-\epsilon^{1/2})\left(\frac{|\partial L(v)|}{r-1} \right)^{\frac{r-1}{r-2}}
\le |L(v)| \le
\left(\frac{|\partial L(v)|}{r-1} \right)^{\frac{r-1}{r-2}}.
\end{align}
On the other hand, (\ref{inequ-d(v)-low-bound-absolute}) and (\ref{inequ-d(v)-up-bound-absolute}) give
\begin{align}\label{inequ-mathcal(H)v-low-bound-absolute}
\left(1-4r^3\epsilon^{1/2}\right)\frac{n^{r-1}}{r^{r-1}}
\le |L(v)| \le
\left(1+4r^4\epsilon^{1/2}\right)\frac{n^{r-1}}{r^{r-1}}.
\end{align}
Combining $(\ref{inequ-mathcal(H)v-low-bound-relative})$ with $(\ref{inequ-mathcal(H)v-low-bound-absolute})$ we obtain
\begin{align}\label{inequ-lemma-partial-mathcal(H)v-low-up-bound}
\left(1 - 5r^{4}\epsilon^{1/2}\right)(r-1)\left(\frac{n}{r}\right)^{r-2}
\le |\partial L(v)| \le
\left(1+5r^{4}\epsilon^{1/2}\right)(r-1)\left(\frac{n}{r}\right)^{r-2}.
\end{align}
Let $x$ be the real number such that $|\partial L(v)| = x^{r-2}/(r-1)^{r-3}$,
and for convenience let us assume that $x$ is an integer.
Then $(\ref{inequ-lemma-partial-mathcal(H)v-low-up-bound})$ implies that
\begin{align}\label{equ:x-low-up-bounds-expansion}
\left(1 - 5r^{4}\epsilon^{1/2}\right) \frac{r-1}{r}n
\le x
\le \left(1+5r^{4}\epsilon^{1/2}\right) \frac{r-1}{r}n.
\end{align}
Now Lemma~\ref{LEMMA:cancellative-KK-stability-reduction} applied to $L(v)$ implies that
there exists a set $U_v \subset N(v) \subset V(\mathcal{H})$ of size $(1\pm \epsilon_1)x$
(and to keep our calculations simple let us assume that $|U_v| = x$) such that
\begin{align}\label{inequ-H-Uv-low-up-bound}
|L(v)[U_v]|
\ge \left(1-\epsilon_1\right)|L(v)|
\ge \left(1-2\epsilon_1\right)\left(\frac{x}{r-1}\right)^{r-1},
\end{align}
and
\begin{align}\label{inequ-partial-H-Uv-low-up-bound}
|\partial \left(L(v)[U_v]\right)|
\ge \left(1-\epsilon_1\right)|\partial L(v)|
\ge \left(1-2\epsilon_1\right)\frac{x^{r-2}}{(r-1)^{r-3}},
\end{align}
where $\epsilon_{1} = 35r^4\epsilon^{1/4}$
(the exponent $1/4$ is due to $\epsilon^{1/2}$ in the first inequality in $(\ref{inequ-mathcal(H)v-low-bound-relative})$).
On the other hand, it follows from $(\ref{inequ-lemma-partial-mathcal(H)v-low-up-bound})$ and
$(\ref{equ:x-low-up-bounds-expansion})$ that
\begin{align}\label{equ:partial-H-Uv-up-bound-x-expansion}
|\partial \left(L(v)[U_v]\right)|
\le |\partial L(v)|
\le \left(1+2\epsilon_1\right)\frac{x^{r-2}}{(r-1)^{r-3}}.
\end{align}
By $(\ref{inequ-H-Uv-low-up-bound})$, $(\ref{inequ-partial-H-Uv-low-up-bound})$,
$(\ref{equ:partial-H-Uv-up-bound-x-expansion})$, and the induction hypothesis,
there exists a partition $U_v = V_1 \cup \cdots \cup V_{r-1}$ such that
all but at most $\delta_{1} x^{r-1}$ edges in $L(v)[U_v]$ have exactly one vertex in each $V_i$,
where $\delta_1 = \delta_1(r-1,2\epsilon_{1})$ is a sufficiently small constant guaranteed by Lemma~\ref{LEMMA:cancellative-KK-stability-after-reduce}.
Let $L'(v) \subset L(v)$ be the collection of edges in $L(v)$ that have exactly one vertex in each $V_i$.
Then by $(\ref{inequ-mathcal(H)v-low-bound-absolute})$ and $(\ref{inequ-H-Uv-low-up-bound})$,
\begin{align}\label{equ:L'(v)-low-bound-expansion}
|L'(v)|
\ge |L(v)[U_v]| - \delta_1 x^{r-1}
\ge (1-\epsilon_1)|L(v)| - \delta_1 x^{r-1}
\ge (1-\delta_2) \frac{n^{r-1}}{r^{r-1}},
\end{align}
where $\delta_2 = 5r^{3}\epsilon^{1/2} + r^{r}\delta_1$.

Let
$$\mathcal{G} = \left\{A\in\partial\mathcal{H}\colon |N(A)| \ge \left(1 - \epsilon^{1/4}\right)\frac{n}{r} \right\}.$$

\begin{claim}\label{claim-low-bound-SL}
We have $|\mathcal{G}| \ge (1-15r^{5}\epsilon^{1/4}) \frac{n^{r-1}}{r^{r-2}}$.
\end{claim}
\begin{proof}[Proof of Claim~\ref{claim-low-bound-SL}]
By Lemma~\ref{Lemma:N(s)-cap-Uv=empty} and $(\ref{equ:x-low-up-bounds-expansion})$,
for every $v \in \widehat{E}$ and every $A \in L(v)$ we have
\begin{align}
|N(A)|
\le |V(\mathcal{H})\setminus N(v)|
\le |V(\mathcal{H})\setminus U_{v}|
\le \left(1 + 5r^{5}\epsilon^{1/2}\right) \frac{n}{r}. \notag
\end{align}
Therefore, by $(\ref{inequ-d(v)-up-bound-absolute})$ and Lemma~\ref{LEMMA:cancellative-clique-in-r-2-shadow}, all but at most
\begin{align}
|\partial\mathcal{H}| - \sum_{v\in \widehat{E}}|L(v)|
\le (1+\epsilon)\frac{n^{r-1}}{r^{r-2}} - r\left(1-4r^3\epsilon^{1/2}\right)\frac{n^{r-1}}{r^{r-1}}
\le 5r^4\epsilon^{1/2}\frac{n^{r-1}}{r^{r-1}} \notag
\end{align}
edges $A \in \partial \mathcal{H}$ satisfy $N(A)\le \left(1 + 5r^{5}\epsilon^{1/2}\right)n/r$.
It follows that
\begin{align}
r|\mathcal{H}|
& = \sum_{A \in \mathcal{G}}N(A) + \sum_{A \in \partial\mathcal{H}\setminus \mathcal{G}}N(A) \notag\\
& \le  |\mathcal{G}|\left(1 + 5r^{5}\epsilon^{1/2}\right) \frac{n}{r}
       + 5r^4\epsilon^{1/2}\frac{n^{r-1}}{r^{r-1}} \cdot n
       + |\partial\mathcal{H}\setminus \mathcal{G}| \left(1 - \epsilon^{1/4}\right)\frac{n}{r} \notag\\
& = |\partial\mathcal{H}|\left(1 - \epsilon^{1/4}\right)\frac{n}{r}
       + |\mathcal{G}|\left(\epsilon^{1/4}+5r^{5}\epsilon^{1/2}\right) \frac{n}{r}
       + 5r^4\epsilon^{1/2}\frac{n^{r}}{r^{r-1}} \notag\\
& \le (1+\epsilon)\frac{n^{r-1}}{r^{r-2}}\left(1 - \epsilon^{1/4}\right)\frac{n}{r}
       + |\mathcal{G}|\left(\epsilon^{1/4}+5r^{5}\epsilon^{1/2}\right)\frac{n}{r}
       + 5r^4\epsilon^{1/2}\frac{n^{r}}{r^{r-1}} \notag\\
& \le \left(1 - \epsilon^{1/4}+\epsilon\right)\frac{n^{r}}{r^{r-1}}
       + |\mathcal{G}|\left(\epsilon^{1/4}+5r^{5}\epsilon^{1/2}\right)
       + 5r^4\epsilon^{1/2}\frac{n^{r}}{r^{r-1}}. \notag
\end{align}
Since $|\mathcal{H}| \ge (1-\epsilon)n^r/r^r$, the inequality above implies
\begin{align}
\left(1 - \epsilon^{1/4}+\epsilon\right)\frac{n^{r}}{r^{r-1}}
       + |\mathcal{G}|\left(\epsilon^{1/4}+5r^{5}\epsilon^{1/2}\right)\frac{n}{r}
       + 5r^4\epsilon^{1/2}\frac{n^{r}}{r^{r-1}}
\le (1-\epsilon)\frac{n^r}{r^{r-1}}. \notag
\end{align}
Therefore,
\begin{align}
|\mathcal{G}|
\ge \frac{\epsilon^{1/4}-2\epsilon - 5r^4\epsilon^{1/2}}{\epsilon^{1/4}+ 5r^{5}\epsilon^{1/2}}\frac{n^{r-1}}{r^{r-2}}
\ge (1-15r^{5}\epsilon^{1/4})\frac{n^{r-1}}{r^{r-2}}. \notag
\end{align}
\end{proof}

Now fix $v\in \widehat{E}$ and let $V_{r} = V(\mathcal{H})\setminus U_v$.
By Lemma~\ref{Lemma:N(s)-cap-Uv=empty}, every edge $A\in L'(v)$ satisfies $N(A) \subset V_{r}$.
By Claim \ref{claim-low-bound-SL} all but at most
\begin{align}
|\partial\mathcal{H}| - |\mathcal{G}|
\le 16r^{5}\epsilon^{1/4}\frac{n^{r-1}}{r^{r-2}} \notag
\end{align}
edges $A \in L'(v)$ satisfy $|N(A)| \ge (1-\epsilon^{1/4})n/r$.
Therefore, by $(\ref{equ:L'(v)-low-bound-expansion})$,
the number of edges in $\mathcal{H}$ that have exactly one vertex in each $V_i$ is at least
\begin{align}
\left(|L'(v)|-16r^{5}\epsilon^{1/4}\frac{n^{r-1}}{r^{r-2}}\right)(1-\epsilon^{1/4})\frac{n}{r}
& \ge  \left((1-\delta_2) \frac{n^{r-1}}{r^{r-1}}-16r^{5}\epsilon^{1/4}\frac{n^{r-1}}{r^{r-2}}\right)(1-\epsilon^{1/4})\frac{n}{r} \notag\\
& \ge (1-\delta_3) \left(\frac{n}{r}\right)^{r}, \notag
\end{align}
where $\delta_3 = \delta_2 + 17r^{6}\epsilon^{1/4}$ (we can choose $\epsilon>0$ to be sufficiently small such that $\delta_3 \le \delta$).
This completes the proof of Lemma~\ref{LEMMA:cancellative-KK-stability-after-reduce}.
\end{proof}

\section{Expansion of cliques}\label{SEC:proof-KK-stability-clique-expansion}

\subsection{Preliminaries}\label{SUBSEC:prelim-clique-expansion}
For an $r$-graph $\mathcal{H}$ the {\em clique set} $\mathcal{K}_{\mathcal{H}}$ of $\mathcal{H}$ is
\begin{align}
\mathcal{K}_{\mathcal{H}} = \left\{A\subset V(\mathcal{H})\colon \left(\partial_{r-2}\mathcal{H}\right)[A] \cong K_{|A|}\right\}. \notag
\end{align}
It was prove in \cite{LM19A} that
\begin{align}
\sigma(S) \le (\ell-r+1)|\partial\mathcal{H}|, \quad \forall S\in \mathcal{K}_{\mathcal{H}}. \notag
\end{align}
Let $z  = z(\mathcal{H})\ge 0$ be the largest real number such that for all $R \in \mathcal{K}_{\mathcal{H}}$ with $|R| \le \ell-1$,
\[\sigma(R) \le \left(\ell-r+1\right)|\partial\mathcal{H}| - \left(\ell-|R|\right)z.\]

The following inequalities can be found in {\cite[Section 5]{LM19A}}.

\begin{lemma}[\cite{LM19A}]\label{lemma-liu-mubayi-inequality}
Let $\mathcal{H}$ be a $\mathcal{K}_{\ell+1}^{r}$-free $r$-graph,
and $R_0 \in \mathcal{K}_{\mathcal{H}}$ be a set of size at most $\ell-1$ with
$\sigma(R_0) = \left(\ell-r+1\right)|\partial\mathcal{H}| - \left(\ell-|R|\right)z$,
where $z = z(\mathcal{H}) \ge 0$ is defined as above.
Then
\begin{align}\label{inequality-H-sigma(S)}
|\mathcal{H}|
\le \frac{\binom{\ell-1}{r-1}^{\frac{r-2}{r-1}}}{r\binom{\ell-1}{r-2}}(r-1)^{\frac{r-2}{r-1}}|\partial\mathcal{H}|^{\frac{r-2}{r-1}}
\left(\sum_{E\in \partial\mathcal{H}}\sigma(E)\right)^{\frac{1}{r-1}},
\end{align}
\begin{align}\label{inequality-sum-sigma(S)}
\sum_{E\in \partial\mathcal{H}}\sigma(E)
\le \left(\ell-r+1\right)\left(|\partial\mathcal{H}|-2z\right)|\partial\mathcal{H}|
+ z^2 \ell - \frac{\left((\ell-r+1)|\partial\mathcal{H}|-z\ell\right)^2}{|R_0|}.
\end{align}
\end{lemma}

\subsection{Proof of Theorem~\ref{THM:stability-expansion-clique-KK-type}}
The proof is similar to the proof of Theorem~\ref{THM:stability-cancellative-KK-type},
but it is simpler because we just need to prove a similar result as Lemma~\ref{LEMMA:cancellative-KK-stability-reduction}
and then we can use Theorem~\ref{THM:old-stability-H} directly
\footnote{One could also use a similar inductive argument to prove a similar result as Lemma~\ref{LEMMA:cancellative-KK-stability-after-reduce}
to avoid using Theorem~\ref{THM:old-stability-H}.}.

\begin{proof}[Proof of Theorem~\ref{THM:stability-expansion-clique-KK-type}]
Fix $\ell\ge r\ge 2$ and $\delta > 0$.
Let $\epsilon>0$ be a sufficiently constant and $x>0$ be a sufficiently large real number.
Let $\mathcal{H}$ be a $\mathcal{K}_{\ell+1}^{r}$-free $r$-graph satisfying the assumptions in
Theorem~\ref{THM:stability-expansion-clique-KK-type}.
Notice that the inequality in $(\ref{equ:expansion-assumption})$ is equivalent to
\begin{align}\label{inequality-H-and-shadow-H}
|\mathcal{H}| \ge (1-\epsilon)\frac{\binom{\ell}{r}}{\binom{\ell}{r-1}^{\frac{r}{r-1}}}|\partial\mathcal{H}|^{\frac{r}{r-1}}.
\end{align}
Let $z = z(\mathcal{H})$ be the same as defined in Section~\ref{SUBSEC:prelim-clique-expansion},
and let $R_0 \in \mathcal{K}_{\mathcal{H}}$ be a set of size at most $\ell-1$ with
\begin{align}
\sigma(R_{0}) = \left(\ell-r+1\right)|\partial\mathcal{H}| - \left(\ell-|R_{0}|\right)z. \notag
\end{align}

\begin{claim}\label{claim-size-sun-sigma(S)}
We have
\begin{align}\label{equ:sum-sigam-shadow}
\sum_{E \in \partial\mathcal{H}}\sigma(E) \le (1-r\epsilon)\frac{(\ell-r+1)(r-1)}{\ell}|\partial\mathcal{H}|^2.
\end{align}
\end{claim}
\begin{proof}[Proof of Claim \ref{claim-size-sun-sigma(S)}]
Suppose to the contrary that $(\ref{equ:sum-sigam-shadow})$ fails.
Then by $(\ref{inequality-H-sigma(S)})$,
\begin{align}
|\mathcal{H}|
& \le \frac{\binom{\ell-1}{r-1}^{\frac{r-2}{r-1}}}{r\binom{\ell-1}{r-2}}(r-1)^{\frac{r-2}{r-1}}|\partial\mathcal{H}|^{\frac{r-2}{r-1}}
     \left(\sum_{E\in \partial\mathcal{H}}\sigma(E)\right)^{\frac{1}{r-1}} \notag\\
& \le (1-r\epsilon)^{\frac{1}{r-1}}
    \frac{\binom{\ell-1}{r-1}^{\frac{r-2}{r-1}}}{r\binom{\ell-1}{r-2}}(r-1)^{\frac{r-2}{r-1}}|\partial\mathcal{H}|^{\frac{r-2}{r-1}}
    \frac{ (\ell-r+1)^{\frac{1}{r-1}} (r-1)^{\frac{1}{r-1}} }{ \ell^{\frac{1}{r-1}} } |\partial\mathcal{H}|^{\frac{2}{r-1}} \notag\\
& = (1-r\epsilon)^{\frac{1}{r-1}} \frac{\binom{\ell}{r}}{\binom{\ell}{r-1}^{\frac{r}{r-1}}}|\partial\mathcal{H}|^{\frac{r}{r-1}}
< (1-\epsilon)\frac{\binom{\ell}{r}}{\binom{\ell}{r-1}^{\frac{r}{r-1}}}|\partial\mathcal{H}|^{\frac{r}{r-1}}, \notag
\end{align}
contradicting $(\ref{inequality-H-and-shadow-H})$.
\end{proof}

Next, we show that $z$ is close to $\frac{\ell-r+1}{\ell}|\partial\mathcal H|$.

\begin{claim}\label{claim-size-z}
We have
\begin{align}\label{equ:z-low-up-bound}
z =  (1 \pm \ell r \epsilon^{1/2})\frac{\ell-r+1}{\ell}|\partial\mathcal{H}|.
\end{align}
\end{claim}
\begin{proof}[Proof of Claim \ref{claim-size-z}]
Suppose to the contrary that $(\ref{equ:z-low-up-bound})$ fails.
Then by $(\ref{inequality-sum-sigma(S)})$,
\begin{align}
\sum_{E \in \partial\mathcal{H}}\sigma(E)
& \le \left(\ell-r+1\right)\left(|\partial\mathcal{H}|-2z\right)|\partial\mathcal{H}|
        + z^2 \ell - \frac{\left((\ell-r+1)|\partial\mathcal{H}|-z\ell\right)}{|R_0|} \notag\\
& = \frac{(\ell-r+1)(r-1)}{\ell}|\partial\mathcal{H}|^2
      - \ell \left(\frac{\ell}{|R_0|}-1\right)\left(z-\frac{\ell-r+1}{\ell}|\partial\mathcal{H}|\right)^2 \notag\\
& \le \frac{(\ell-r+1)(r-1)}{\ell}|\partial\mathcal{H}|^2
     - \ell \left(\frac{\ell}{\ell-1}-1\right)\left(\ell r \epsilon^{1/2}\cdot \frac{\ell-r+1}{\ell}|\partial\mathcal{H}|\right)^2 \notag\\
& < (1- r\epsilon)\frac{(\ell-r+1)(r-1)}{\ell}|\partial\mathcal{H}|^2 \notag
\end{align}
contradicting Claim~\ref{claim-size-sun-sigma(S)}.
\end{proof}

It follows from the definition of $z$ and Claim \ref{claim-size-z} that
for every $R \in \mathcal{K}_{\mathcal{H}}$ with $|R| \le \ell-1$,
\begin{align}
\sigma(R)
& \le \left(\ell-r+1\right)|\partial\mathcal{H}| - \left(\ell-|R|\right)z \notag\\
& \le \left(\ell-r+1\right)|\partial\mathcal{H}| - \left(\ell-|R|\right)
           (1-\ell r \epsilon^{1/2})\frac{\ell-r+1}{\ell}|\partial\mathcal{H}| \notag\\
& \le (1+\ell^2 r \epsilon^{1/2})\frac{(\ell-r+1)|R|}{\ell}|\partial\mathcal{H}|. \notag
\end{align}
In particular, for every $v \in V(\mathcal{H})$ we have
\begin{align}\label{equ-upper-max-degree-v}
d(v) = \sigma(\{v\})
\le (1+\ell^2 r \epsilon^{1/2})\frac{\ell-r+1}{\ell}|\partial\mathcal{H}|,
\end{align}
and for every $E \in \partial\mathcal{H}$ we have
\begin{align}\label{equ-upper-sigma-shadow-E}
\sigma(E)
\le (1+\ell^2 r \epsilon^{1/2})\frac{(\ell-r+1)(r-1)}{\ell}|\partial\mathcal{H}|.
\end{align}

Let
$$\mathcal{G} = \left\{E \in \partial\mathcal{H}\colon
                       \sigma(E) \ge (1-\epsilon^{1/4}) \frac{(\ell-r+1)(r-1)}{\ell}|\partial\mathcal{H}|\right\}.$$

\begin{claim}\label{claim-size-G}
We have $|\mathcal{G}| \ge (1-\ell^2 r \epsilon^{1/4})|\partial\mathcal{H}|$.
\end{claim}
\begin{proof}[Proof of Claim \ref{claim-size-G}]
It follows from Lemma~\ref{LEMMA:function-concentrate},  Claim~\ref{claim-size-sun-sigma(S)}, and
$(\ref{equ-upper-sigma-shadow-E})$ that
\begin{align}
|\partial\mathcal{H}\setminus \mathcal{G}|
& \le \frac{(1+\ell^2 r \epsilon^{1/2})(\ell-r+1)(r-1)|\partial\mathcal{H}|/\ell
        -\sum_{E\in\partial\mathcal{H}}\sigma(E)/|\partial\mathcal{H}|}
        {(1+\ell^2 r \epsilon^{1/2})(\ell-r+1)(r-1)|\partial\mathcal{H}|/\ell
          -(1-\epsilon_4)(\ell-r+1)(r-1)|\partial\mathcal{H}|/\ell}|\partial\mathcal{H}| \notag\\
& \le \frac{\ell^2 r \epsilon^{1/2} + r\epsilon}{\ell^2 r \epsilon^{1/2} + \epsilon^{1/4}} |\partial\mathcal{H}|
\le \ell^2 r \epsilon^{1/4}|\partial\mathcal{H}|. \notag
\end{align}
\end{proof}

Let $U = \partial_{r-2}\mathcal{G} \subset V(\mathcal{H})$.

\begin{claim}\label{claim-d(u)-lower}
For every $u \in U$ we have $d(u) \ge (1-2\epsilon^{1/4})\frac{\ell-r+1}{\ell}|\partial\mathcal{H}|$.
\end{claim}
\begin{proof}[Proof of Claim \ref{claim-d(u)-lower}]
It follows from the definition of $U$ that for every $u \in U$ there exists $E\in \mathcal{G}$ with $u \in E$.
Then it follows from the definition of $\mathcal{G}$ that
$\sigma(E) \ge (1-\epsilon^{1/4}) (\ell-r+1)(r-1)|\partial\mathcal{H}|/\ell$.
So by $(\ref{equ-upper-max-degree-v})$
\begin{align}
d(u)
& = \sigma(E) - \sum_{v \in E\setminus \{u\}} d_{{\mathcal{H}}}(v) \notag\\
& \ge (1-\epsilon^{1/4}) \frac{(\ell-r+1)(r-1)}{\ell}|\partial\mathcal{H}|
            - (r-2)(1+\ell^2 r \epsilon^{1/2})\frac{\ell-r+1}{\ell}|\partial\mathcal{H}| \notag\\
& \ge (1-2\epsilon^{1/4})  \frac{\ell-r+1}{\ell}|\partial\mathcal{H}|. \notag
\end{align}
\end{proof}

Next we show an upper bound for $|U|$.

\begin{claim}\label{claim-size-U-upper-bound}
We have $|U| \le (1+4\epsilon^{1/4})\ell \left(\frac{|\partial\mathcal{H}|}{\binom{\ell}{r-1}}\right)^{1/(r-1)}$.
\end{claim}
\begin{proof}[Proof of Claim~\ref{claim-size-U-upper-bound}]
Since $\sum_{u \in U}d(u) \le r|{\mathcal{H}}|$,
by Claim \ref{claim-d(u)-lower} and Theorem~\ref{THM:LM-clique-expansion},
\begin{align}
|U|
 \le \frac{r|\mathcal{H}|}{(1-2\epsilon^{1/4})\frac{\ell-r+1}{\ell}|\partial\mathcal{H}|}
& \le \frac{r\binom{\ell}{r}\left(\frac{|\partial\mathcal{H}|}{\binom{\ell}{r-1}}\right)^{\frac{r}{r-1}}}
        {(1-2\epsilon^{1/4})\frac{\ell-r+1}{\ell}|\partial\mathcal{H}|} \notag\\
& =\frac{\frac{r}{\ell-r+1}\frac{\binom{\ell}{r}}{\binom{\ell}{r-1}}}{(1-2\epsilon^{1/4})}
      \ell\left(\frac{|\partial\mathcal{H}|}{\binom{\ell}{r-1}}\right)^{1/(r-1)}
\le (1+4\epsilon^{1/4})\ell \left(\frac{|\partial\mathcal{H}|}{\binom{\ell}{r-1}}\right)^{1/(r-1)}. \notag
\end{align}
Here we used the identity $\binom{\ell}{r}/\binom{\ell}{r-1} = (\ell-r+1)/r$.
\end{proof}

\begin{claim}\label{claim-size-hat-G}
We have $|\mathcal{H}[U]| \ge (1-9\ell^{2r} r \epsilon^{1/4})\left(\frac{|\partial\mathcal{H}|}{\binom{\ell}{r-1}}\right)^{\frac{r}{r-1}}$.
\end{claim}
\begin{proof}[Proof of Claim \ref{claim-size-hat-G}]
By Claims~\ref{claim-d(u)-lower} and~\ref{claim-size-G}, for every $u \in U$ we have
\begin{align}
d_{\mathcal{H}[U]}(u)
  \ge d_{\mathcal{H}}(u) - |\partial\mathcal{H}\setminus{\mathcal{G}}|
& \ge (1-2\epsilon^{1/4})\frac{\ell-r+1}{\ell}|\partial\mathcal{H}| - \ell^2 r \epsilon^{1/4} |\partial\mathcal{H}| \notag\\
& \ge (1-3\ell^3 r \epsilon^{1/4})\frac{\ell-r+1}{\ell}|\partial\mathcal{H}|. \notag
\end{align}
For every $0 \le i \le r$ let $\mathcal{E}_{i}$ be the set of edges in $\mathcal{H}$ that have exactly $i$ vertices in $U$
and note that $\mathcal{E}_r = \mathcal{H}[U]$.
Then by $(\ref{equ-upper-max-degree-v})$ and Claim~\ref{claim-size-U-upper-bound},
\begin{align}
\sum_{i\in[r-1]}i|\mathcal{E}_i|
= \sum_{u\in U}d_{\mathcal{H}}(u) - r|\mathcal{E}_r|
& = \sum_{u\in U}d_{\mathcal{H}}(u) - \sum_{u\in U}d_{\mathcal{H}[U]}(u) \notag\\
& \le \sum_{u\in U}\left(\Delta(\mathcal{H}) - d_{\mathcal{H}[U]}(u)\right) \notag\\
& \le (\ell^2 r \epsilon^{1/2} + 3\ell^3 r \epsilon^{1/4}) \frac{\ell-r+1}{\ell}|\partial\mathcal{H}| |U| \notag\\
& \le 4\ell^3 r \epsilon^{1/4} \frac{\ell-r+1}{\ell}|\partial\mathcal{H}| \cdot
      (1+4\epsilon^{1/4})\ell \left(\frac{|\partial\mathcal{H}|}{\binom{\ell}{r-1}}\right)^{1/(r-1)} \notag\\
& \le 8\ell^{3} r^{2} \binom{\ell}{r} \epsilon^{1/4} \left(\frac{|\partial\mathcal{H}|}{\binom{\ell}{r-1}}\right)^{\frac{r}{r-1}}. \notag
\end{align}
On the other hand, by Claim~\ref{claim-size-G} and Theorem~\ref{THM:LM-cancellative},
\begin{align}
|\mathcal{E}_0|
\le \binom{\ell}{r}\left(\frac{|\partial\mathcal{H}|-|\mathcal{G}|}{\binom{\ell}{r-1}}\right)^{\frac{r}{r-1}}
\le \ell^2 r\epsilon^{1/4} \left(\frac{|\partial\mathcal{H}|}{\binom{\ell}{r-1}}\right)^{\frac{r}{r-1}}. \notag
\end{align}
Therefore,
\begin{align}
|\mathcal{H}[U]|
= |\mathcal{H}| - \sum_{i=0}^{r-1}|\mathcal{E}_{i}|
& \ge (1-\epsilon)\binom{\ell}{r}\left(\frac{|\partial\mathcal{H}|}{\binom{\ell}{r-1}}\right)^{\frac{r}{r-1}}
    -\left(8\ell^{3} r^{2} \binom{\ell}{r} \epsilon^{1/4}+\ell^2 r\epsilon^{1/4}\right)
     \left(\frac{|\partial\mathcal{H}|}{\binom{\ell}{r-1}}\right)^{\frac{r}{r-1}} \notag\\
& \ge (1-9\ell^{3} r^{2} \epsilon^{1/4})\binom{\ell}{r}\left(\frac{|\partial\mathcal{H}|}{\binom{\ell}{r-1}}\right)^{\frac{r}{r-1}}. \notag
\end{align}
\end{proof}

Let $m = |U|$.
Then it follows from Claim~\ref{claim-size-U-upper-bound} and $|\partial\mathcal{H}| = \binom{\ell}{r-1}(x/\ell)^{r-1}$ that
\begin{align}
m
\le (1+4\epsilon^{1/4})\ell \left(\frac{|\partial\mathcal{H}|}{\binom{\ell}{r-1}}\right)^{1/(r-1)}
\le (1+4\epsilon^{1/4})x.  \notag
\end{align}
Claim~\ref{claim-size-hat-G} implies that
\begin{align}
|\mathcal{H}[U]|
& \ge (1-9\ell^{3} r^{2} \epsilon^{1/4})\binom{\ell}{r}\left(\frac{|\partial\mathcal{H}|}{\binom{\ell}{r-1}}\right)^{\frac{r}{r-1}} \notag\\
& \ge (1-9\ell^{3} r^{2} \epsilon^{1/4})\binom{\ell}{r}\left(\frac{x}{\ell}\right)^{r}\notag\\
& \ge (1-9\ell^{3} r^{2} \epsilon^{1/4})\frac{1}{(1+4\epsilon^{1/4})^{r}}\binom{\ell}{r}\left(\frac{m}{\ell}\right)^{r}
 \ge (1-10\ell^{3} r^{2} \epsilon^{1/4})\binom{\ell}{r}\left(\frac{m}{\ell}\right)^{r}. \notag
\end{align}
Now Theorem~\ref{THM:old-stability-H} applied to $\mathcal{H}[U]$ implies that
$\mathcal{H}[U]$ contains a subgraph $\mathcal{H}'$ of size at least
\begin{align}
|\mathcal{H}[U]| - \delta_1 m^r
\ge (1-10\ell^{3} r^{2} \epsilon^{1/4})\binom{\ell}{r}\left(\frac{m}{\ell}\right)^{r} -\delta_1 m^r
& \ge |\mathcal{H}| - (10\ell^{3} r^{2} \epsilon^{1/4}+\delta_1)m^r \notag\\
& \ge |\mathcal{H}| - 2(10\ell^{3} r^{2} \epsilon^{1/4}+\delta_1)x^r \notag
\end{align}
such that $\mathcal{H}'$ is a subgraph of $T_{r}(m,\ell)$.
If $m \le \lceil x \rceil$, then let $V' = U$ and we are done.
Otherwise let $V' \subset U$ be a subset of size $\lceil x \rceil$.
Then due to $m \le (1+4\epsilon^{1/4})x$, the number of edges in $\mathcal{H}'[V']$ is at least
$|\mathcal{H}'| - 4\epsilon^{1/4} x^{r} \ge |\mathcal{H}| - 2(10\ell^{3} r^{2} \epsilon^{1/4}+\delta_1)x^r - 4\epsilon^{1/4} x^{r}$
(we can choose $\epsilon>0$ to be sufficiently small such that
$2(10\ell^{3} r^{2} \epsilon^{1/4}+\delta_1)+ 4\epsilon^{1/4} \le \delta$).
This completes the proof of Theorem~\ref{THM:stability-expansion-clique-KK-type}.
\end{proof}

\section{Concluding remarks}\label{SEC:remarks}
In this work we showed some extensions of Keevash's  stability result of the Kruskal-Katona theorem
to the classes of calcellative hypergraphs and hypergraphs without the expansion of cliques.
In general, one could ask whether similar results hold for other $\mathcal{F}$-free hypergraphs.
An classical example suggested by S\'{o}s is the Fano plane, which is the $3$-graph on vertex set $[7]$
with edge set
\begin{align}
\{123, 345, 561, 174, 275, 376, 246\}. \notag
\end{align}
The extremal properties of the Fano plane have been well studies by several authors (see e.g. \cite{DF00,KS05,FS05,BR19}).
However, a Kruskal-Katona type result for the Fano plane is still not known.

It is interesting that the inequality in Theorem~\ref{THM:LM-cancellative} is tight for every integer $r\ge 2$ while
the maximum size of an $n$-vertex cancellative $r$-graph is still unknown (even asymptotically) for every $r\ge 5$.
This suggests that one could prove a Kruskal-Katona type result for a family $\mathcal{F}$ whose Tur\'{a}n density is not known.
A famous example is the complete $3$-graph on four vertices $K_{4}^{3}$ (see \cite{TU41}).
It was shown in \cite{LM19A} that the Tur\'{a}n problem for $K_{4}^{3}$ does not have the stability property
(assuming the famous Tur\'{a}n tetrahedron conjecture is true).
So, it would be very interesting if the Kruskal-Katona type result for $K_{4}^{3}$ has the stability property.
There are many other interesting cases one could consider for $\mathcal{F}$,
and we refer the reader to the nice survey of Keevash \cite{KE11} for more details.

It seems that Theorems~\ref{THM:stability-cancellative-KK-type} and~\ref{THM:stability-expansion-clique-KK-type}
belong to a new type of stability results for $\mathcal{F}$-free hypergrahs,
which are different from the stability results proved before
(see e.g. \cite{SI68,AES74,KM04,FS05,KS05,MU06,PI08,HK13,NY17,BIJ17,NY18,BNY19,LM19,LMR1,LMR2}).
It would be interesting to see whether there are any applications of this new type of stability results.
\bibliographystyle{abbrv}
\bibliography{stabilityExpansion}
\end{document}